\documentclass[11pt]{amsart}

\usepackage{array,float,hyperref,yhmath, tikz-cd}
\usepackage{amssymb, amsmath, amsthm, amsbsy, amscd, mathrsfs,stmaryrd}
\usepackage[normalem]{ulem}
\usepackage{graphicx}
\usepackage{tikz}
\usetikzlibrary{arrows}

\numberwithin{equation}{section}
\numberwithin{figure}{section}

\newcommand{\rev}[1]{{\color{black}{#1}}}

\theoremstyle{plain}
 \newtheorem{theorem}{Theorem}[section]
 \newtheorem{proposition}[theorem]{Proposition}
 \newtheorem{lemma}[theorem]{Lemma}
 \newtheorem{corollary}[theorem]{Corollary}
 
 \newtheorem{conjecture}[theorem]{Conjecture}

\theoremstyle{definition}
 \newtheorem{definition}[theorem]{Definition}

\theoremstyle{remark}

 \newtheorem{remark}[theorem]{Remark}
 \newtheorem*{acknowledgements}{Acknowledgements}
 
 \newtheorem{example}[theorem]{Example}

\newcommand{\be}{\begin{enumerate}}
\newcommand{\ee}{\end{enumerate}}

\newcommand{\brem}{\begin{remark}}
\newcommand{\erem}{\end{remark}}

\newcommand{\bp}{\begin{proof}}
\newcommand{\ep}{\end{proof}}

\newcommand{\Fq}{\ensuremath{\mathbb{F}_q}}

\newcommand{\N}{\ensuremath{\mathbb{N}}}

\newcommand{\C}{\ensuremath{\mathbb{C}}}
\newcommand{\R}{\ensuremath{\mathbb{R}}}
\newcommand{\Q}{\ensuremath{\mathbb{Q}}}
\newcommand{\Z}{\ensuremath{\mathbb{Z}}}

\newcommand{\bfr}{\ensuremath{\mathbf{r}}}

\newcommand{\bfx}{{\boldsymbol{x}}}
\newcommand{\bfy}{\ensuremath{\mathbf{y}}}

\newcommand{\mcB}{\mathcal{B}}

\newcommand{\mcO}{\mathcal{O}}
\newcommand{\mcP}{\mathcal{P}}

\newcommand{\set}[1]{\lbrace#1\rbrace}
\newcommand{\abs}[1]{\left|#1\right|}

\newcommand{\rmN}{\textup{N}}

\newcommand{\rk}{\ensuremath{{\rm rk}}}

\newcommand{\tr}{\textup{t}}
\newcommand{\lri}{\mathfrak o}

\newcommand{\la}{\langle}
\newcommand{\ra}{\rangle}

\newcommand{\Gri}{\ensuremath{\mathcal{O}}}

\newcommand{\mfp}{\mathfrak{p}}

\newcommand{\rarr}{\rightarrow}

\newcommand{\bj}{\overline{\jmath}}
\newcommand{\sll}{\mathfrak{sl}}
\newcommand{\nch}{\varepsilon_n}
\newcommand{\udots}{\mathinner{\mskip1mu\raise1pt\vbox{\kern7pt\hbox{.}}
\mskip2mu\raise4pt\hbox{.}\mskip2mu\raise7pt\hbox{.}\mskip1mu}}

\newcommand{\level}{N}

\DeclareMathOperator{\nega}{neg}
\DeclareMathOperator{\nsp}{nsp}

\DeclareMathOperator{\Neg}{Neg}
\DeclareMathOperator{\rmaj}{rmaj}
\DeclareMathOperator{\des}{des}
\DeclareMathOperator{\gp}{gp}
\DeclareMathOperator{\Des}{Des}

\DeclareMathOperator{\Spec}{Spec}

\DeclareMathOperator{\SL}{SL}

\DeclareMathOperator{\inv}{inv}
\DeclareMathOperator{\invv}{inv}

\DeclareMathOperator{\im}{im}
\DeclareMathOperator{\Mat}{Mat}
\DeclareMathOperator{\Alt}{Alt}
\DeclareMathOperator{\Sym}{Sym}

\DeclareMathOperator{\Tr}{Tr}
\renewcommand{\epsilon}{\varepsilon}
\renewcommand{\phi}{\varphi}

\makeatletter
\newcommand{\pushright}[1]{\ifmeasuring@#1\else\omit\hfill$\displaystyle#1$\fi\ignorespaces}

\begin{document}

\title[Enumerating traceless matrices]{Enumerating traceless matrices over compact discrete valuation
  rings}

\date{\today} \author{Angela Carnevale, Shai Shechter and Christopher
  Voll} \address{[AC,CV] Fakult\"at f\"ur Mathematik, Universit\"at
  Bielefeld\\ Postfach 100131\\ D-33501 Bielefeld\\Germany}
\email{acarneva1@math.uni-bielefeld.de, C.Voll.98@cantab.net}

\address{[ShS] Department of Mathematics\\ Ben Gurion University of the Negev\\ Beer-Sheva 84105\\Israel}
\email{shais@post.bgu.ac.il}




\begin{abstract}
  We enumerate traceless square matrices over finite quotients of
  compact discrete valuation rings by their image sizes. We express
  the associated rational generating functions in terms of statistics
  on symmetric and hyperoctahedral groups, viz.\ Coxeter groups of
  types $A$ and~$B$, respectively. These rational functions may also
  be interpreted as local representation zeta functions associated to
  the members of an infinite family of finitely generated
  class-$2$-nilpotent groups.

  As a byproduct of our work, we obtain descriptions of the numbers of
  traceless square matrices over a finite field of fixed rank in terms
  of statistics on the hyperoctahedral groups.
\end{abstract}

\maketitle

\thispagestyle{empty}

\section{Introduction and statement of main results}
Let $n\in\N$. Given a ring $R$, we write $\sll_n(R)$ for the set of
$n\times n$-matrices over $R$ of trace zero. In the case that $R$ is a
finite field $\Fq$, we write, for $i\in\{0,1,\dots,n\}$,
$$B_{n,n-i}(\Fq) = \{ \bfx \in \sll_n(\Fq) \mid \rk(\bfx) = n-i\}$$ for the
set of traceless $n\times n$-matrices over $\Fq$ of
rank~$n-i$. Formulae expressing the cardinalities of these sets as
polynomials in the field's cardinality $q$ are well-known, e.g.\ by
the work \cite{Bender/74} of Bender; cf.\ Lemma~\ref{lem:bender}.

In the current paper, we generalize these enumerative formulae to
traceless matrices over more general finite quotients of compact
discrete valuation rings. As a byproduct we obtain an interpretation
of the polynomials $|B_{n,n-i}(\Fq)|$ in terms of statistics on
hyperoctahedral groups, viz.\ finite Coxeter groups of type $B$;
cf.\ Proposition~\ref{pro:bender.coxeter}.

\subsection{Counting traceless matrices over compact discrete
  valuation rings}
We now state our generalized counting problem. Let $\lri$ be a compact
discrete valuation ring of arbitrary characteristic, with unique
maximal ideal $\mfp$, residue field cardinality $q$, and residue field
characteristic $p$. Given a \emph{level} $\level\in\N_0$, we set
$\lri_\level = \lri/\mfp^\level$.  Clearly, traceless matrices of
level $\level=1$ decompose in a disjoint union as follows:
$$\sll_n(\lri_1) = \bigcup_{i=0}^{n} B_{n,n-i}(\Fq).$$ Noting that
$\bfx\in\sll_n(\lri_1)$ has rank $\rk(\bfx)=n-i$ if and only if
$|\im(\bfx)|=q^{n-i}$ when $\bfx$ is viewed as an endomorphism of
$\lri_1^n$ suggests a natural generalization of this decomposition to
arbitrary levels, viz.\ by image sizes. For any $\level\in\N$ and
$\bfx\in\sll_n(\lri_\level)$, write $\im(\bfx)$ for the image of
$\bfx$, viewed as an endomorphism of $\lri_\level^n$. For $\level\in\N$, we write
$\sll_n(\lri_\level)^* = \sll_n(\lri_\level) \setminus (\mfp \cdot
\sll_n(\lri_\level))$
for the \emph{primitive} traceless $n\times n$-matrices of
level~$\level$, i.e.\ those which are not zero modulo~$\mfp$. We set
$\sll_n(\lri_0)^*=\{0\}$.

\begin{definition} The \emph{image zeta function of $\sll_n(\lri)$} is
  the ordinary generating function
  $$ \mathcal{P}_{n,\lri}(s) := \sum_{\level =
    0}^\infty\sum_{\bfx\in\sll_n(\lri_\level)^*} |\im(\bfx)|^{-s} \in \Q
  \llbracket q^{-s} \rrbracket,$$ where $s$ is a complex variable.
\end{definition}

The series $\mcP_{n,\lri}(s)$ is well-defined as summation is
restricted to primitive matrices.

One of our main results shows that the power series $\mcP_{n,\lri}(s)$
is a rational function over $\Q$ in $q$ and $q^{-s}$, and describes
this rational function explicitly. In the following, $S_n$ denotes the
symmetric group of degree $n$, whereas $\ell$ and $\Des$ denote the
standard Coxeter length and the descent set statistics on~$S_n$,
respectively; see Section~\ref{subsec:coxeter.prelim} for details.

\begin{theorem}\label{thm:main}
 For all compact discrete valuation rings $\lri$, with residue field
 cardinality~$q$, say, the image zeta function of $\sll_n(\lri)$
 satisfies
$$\mcP_{n,\lri}(s) = \left( \sum_{w\in
    S_n}q^{-\ell(w)}\prod_{j\in\Des(w)}q^{n^2-j^2-1-s(n-j)}\right)\prod_{j=0}^{n-1}\frac{1-q^{j-s}}{1-q^{n^2-j^2-1-s(n-j)}}.$$
\end{theorem}

\begin{example}\label{exa:small.n}  Throughout, we write $t=q^{-s}$.
\begin{enumerate}
\item
$$\mcP_{1,\lri}(s) = 1.$$
\item 
  \begin{equation*}\label{exa:n=2}\mcP_{2,\lri}(s) = (1+qt) \frac{(1-t)(1-qt)}{(1-q^2t)(1-q^3t^2)} = \frac{(1-t)(1-q^2t^2)}{(1-q^2t)(1-q^3t^2)}.
    \end{equation*}
\item
\begin{align*}\mcP_{3,\lri}(s) &= (1+q^2t)(1+q^3t+q^6t^2)\frac{(1-t)(1-qt)(1-q^2t)}{(1-q^4t)(1-q^7t^2)(1-q^8t^3)} \\ &= \frac{(1-q^4t^2)(1-q^9t^3)(1-t)(1-qt)(1-q^2t)}{(1-q^2t)(1-q^3t)(1-q^4t)(1-q^7t^2)(1-q^8t^3)}.
\end{align*}
\item
\begin{equation*}
  \mathcal{P}_{4,\lri}(s) = V_4(q,t)
  \frac{(1-t)(1-qt)(1-q^2t)(1-q^3t)}{(1-q^6t)(1-q^{11}t^2)(1-q^{14}t^3)(1-q^{15}t^4)},
\end{equation*}
where $V_4(q,t) = (1+q^4t)V'_4(q,t)$ and
\begin{multline*}
  V'_4(X,Y) =
  X^{21}Y^5+X^{18}Y^4+X^{16}Y^4+X^{13}Y^3+X^{12}Y^3+X^{11}Y^3+\\X^{10}Y^2+X^9Y^2+X^8Y^2+X^5Y+X^3Y+1.
\end{multline*}
\rev{Computations with {\sf SageMath} show that the polynomial
  $V'_4(X,Y)\in \Q[X,Y]$ is irreducible.}
\item 
$$\mcP_{5,\lri}(s) = V_5(q,t) \frac{(1-t)(1-qt)(1-q^2t)(1-q^3t)(1-q^4t)}{(1-q^8t)(1-q^{15}t^2)(1-q^{20}t^3)(1-q^{23}t^4)(1-q^{24}t^5)},$$
where $V_5(X,Y)\in\Q[X,Y]$ has degree $56$ in $X$ and $10$ in
$Y$\rev{. Computations with {\sf SageMath} show that $V_5(X,Y)$ is
  irreducible.}
\end{enumerate}
\end{example}

To prove Theorem~\ref{thm:main}---itself a reformulation of
Theorem~\ref{thm:Pn.factor}---we first organize the enumeration of
primitive traceless matrices of given level by the matrices'
elementary divisor types;
cf.\ Section~\ref{subsec:counting.edt}. Proposition~\ref{pro:Kn-to-Gil}
reduces this problem to the problems of enumerating the sets
$B_{n,n-i}(\Fq)$ and enumerating sets of arbitrary (viz.\ not
necessarily traceless) matrices of smaller dimension and level. A
first formula for $\mcP_{n,\lri}(s)$ is obtained in
Theorem~\ref{thm:Pn.first} by combining the solutions of the former
problem by Bender's formula (cf.\ Lemma~\ref{lem:bender}) and the
latter problem by \cite[Proposition~3.4]{StasinskiVoll/14}.

\subsection{Counting traceless matrices over finite fields via (signed) permutation
  statistics} A key step in the proof of Theorem~\ref{thm:main} is the
formulation of the numbers $|B_{n,n-i}(\Fq)|$ in terms of statistics
on signed permutation groups. In the following result,---which follows
directly by combining Lemmas~\ref{lem:bender} and
\ref{lem:bender.cox}---we denote by $B_n$ the hyperoctahedral (or
signed permutation) group of degree~$n$. The precise definitions of
this group, its \emph{quotients} $B_n^{\{i\}^c}$ as well as the
statistics $\nega$, $\nch$, and $\ell$ on $B_n$ are given in
Section~\ref{subsec:coxeter.prelim}.

\begin{proposition}\label{pro:bender.coxeter}
  For $i\in [n-1]_0$,
  $$|B_{n,n-i}(\Fq)| = q^{n^2-i^2-1} \sum_{w\in B_n^{\{i\}^c}}
  (-1)^{\nega(w)}q^{(\nch -\ell)(w)}.$$
\end{proposition}

Proposition~\ref{pro:bender.coxeter} complements similar results on
the numbers of square matrices over finite fields of fixed ranks which
are symmetric, antisymmetric, or satisfy no further restrictions on
their entries; cf.\ Remark~\ref{rem:similar}. It allows us to rewrite
our ``first formula'' for the image zeta function~$\mcP_{n,\lri}(s)$
given in Theorem~\ref{thm:Pn.first} in terms of a generating
polynomial on the Weyl group~$B_n$, leading to a ``second formula''
for $\mcP_{n,\lri}(s)$ in Proposition~\ref{pro:Pn.second}. That this
generating polynomial on $B_n$ factorizes as a product of
$\prod_{j=0}^{n-1}(1-q^{j-s})$ and a generating polynomial on $S_n$ is
a consequence of the general Proposition~\ref{pro:factor},
establishing a factorization of a generating polynomial controlling
the joint distribution of four statistics on $B_n$.  Given this
factorization, Theorem~\ref{thm:Pn.factor} and hence
Theorem~\ref{thm:main} follow swiftly.

\subsection{Applications to representation zeta functions of nilpotent groups}\label{subsec:app.T}
The Poincar\'e series $\mathcal{P}_{n,\lri}(s)$ have an interpretation
as local representation zeta functions of a unipotent group
scheme. Let, more precisely, $K_n$ denote the unipotent group scheme
over $\Q$ associated to the $\Z$-Lie lattice
\begin{equation}\label{equ:rel.Kn}
  \mathcal{K}_n = \la x_1,\dots,x_{2n},y_{ij}, 1 \leq i,j \leq n \mid
  [x_i,x_{n+j}] - y_{ij}, \Tr(\bfy) \ra_{\Z};
\end{equation}
see \cite[Section~2.1.2]{StasinskiVoll/14} and compare with the
unipotent group schemes $F_{n,\delta}$, $G_n$, and $H_n$ defined
analogously in \cite[Section~1.3]{StasinskiVoll/14}. \rev{In
  particular, we follow the convention adopted there that products
  among generators other than those following---by antisymmetry or the
  Jacobi identity---from the given ones are assumed to be trivial.}

We refer to the
finitely generated, class-$2$-nilpotent groups of the form
$K_n(\Gri)$, where $\mcO$ is the ring of integers of a number field,
as \emph{groups of type~$K$}. The \emph{representation zeta function}
$\zeta_{K_n(\Gri)}(s)$ enumerates the irreducible finite-dimensional
complex representations of $K_n(\Gri)$ up to twists by one-dimensional
representations; cf.\ \cite[Section~1.1]{StasinskiVoll/14}. By
\cite[Proposition~2.2]{StasinskiVoll/14}, it satisfies an Euler
product of the form
\begin{equation}\label{equ:euler.K}
  \zeta_{K_n(\Gri)}(s) = \prod_{\mfp\in\Spec(\Gri)}
  \zeta_{K_n(\Gri_{\mfp})}(s),
\end{equation}
where $\mfp$ ranges over the nonzero prime ideals of $\Gri$.  By
design of the relations in~\eqref{equ:rel.Kn} and
\cite[Proposition~2.18]{StasinskiVoll/14},
$$\zeta_{K_n(\Gri_{\mfp})}(s) = \mcP_{n,\Gri_{\mfp}}(s)$$ for all
nonzero prime ideals $\mfp$ of $\Gri$. The formula in
Theorem~\ref{thm:main} thus yields an explicit expression for the
local representation zeta functions of groups of type $K$, in analogy
to \cite[Theorems~B and C]{StasinskiVoll/14}. We note, however, that,
in contrast to their siblings of type $F$, $G$, and $H$, groups of
type $K$ do not, in general, have representation zeta functions that
factorize completely as finite products of translates of Dedekind zeta
functions and their inverses; cf.\ Example~\ref{exa:small.n}. In
Section~\ref{sec:rep.zeta.nilpotent} we deduce some fundamental
properties of the Euler product~\eqref{equ:euler.K}, such as its
abscissa of convergence and some meromorphic continuation. In
Section~\ref{subsec:top} we study the topological representation zeta
functions associated to the group schemes~$K_n$.

\begin{remark}
  The $\Z$-Lie lattice $\mathcal{K}_2$ is a $\Z$-form of the
  7-dimensional complex Lie algebra called $3,7_D$ on
  \cite[p.~483]{Seeley/93}. Both the generic local representation zeta
  functions of the groups of type $K_2$ (cf.\
  Example~\ref{exa:small.n}~(2)) and their topological counterparts
  (cf.\ Proposition~\ref{pro:topo}) have first been computed by means
  of the computer-algebra package $\mathsf{Zeta}$; cf.\
  \cite{RossmannZeta}. We thank T.\ Rossmann for pointing this out to
  us.
\end{remark}

\subsection{Related results}
Counting matrices over finite fields of given rank satisfying some
(linear) restrictions on their entries is a classical theme that
continues to pique researchers' interest to this day.  The paper
\cite{LLMPSZ/11}, for instance, studies matrices over finite fields of
fixed rank with specific entries equal to zero. One particularly rich
context for the study of such questions is work in the wake of a
speculation of Kontsevich, disproved by Belkale and Brosnan in
\cite{BelkaleBrosnan/03}, about the polynomiality in $q$ of the
numbers of $\Fq$-rational points of certain such rank-varieties
arising from finite graphs.

As explained above in the special case of traceless matrices, every
rank-distribution problem for matrices satisfying ($\Q$-)linear
restrictions over finite fields has analogues over finite quotients
$\lri_\level = \lri/\mfp^\level$ of compact discrete valuation
rings~$\lri$: one may enumerate (primitive) such matrices over these
finite rings by the sizes of their images, or---more finely---by their
elementary divisor types. By general, deep results, the ordinary
generating functions---or Poincar\'e series---encoding these numbers,
for $\level\in\N_0$, as coefficients of powers of a variable
$t=q^{-s}$, say, are rational functions in $t$, whose denominators are
finite products of terms of the form $1-q^bt^a$ for some $a\in\N$,
$b\in\N_0$. It is, however, in general a hard problem to compute these
rational functions or only just to describe how they depend on the
local rings~$\lri$. The work \cite{BelkaleBrosnan/03} suggests that
rationality in $q^{-s}$ \underline{and} $q$ as exhibited in
Theorem~\ref{thm:main} for traceless matrices (and in
\cite[Theorem~B]{StasinskiVoll/14} for antisymmetric, symmetric and
general square matrices) should be the exception rather than the
rule. The functional equations
$$\left.\mcP_{n,\lri}(s)\right|_{q\rarr q^{-1}} = q^{n^2-1} \mcP_{n,\lri}(s),$$
however, are general features of Poincar\'e series enumerating
matrices of linear forms by elementary divisor types; cf.\
\cite[Proposition~2.2]{Voll/10}.

A common feature observed for the four families of matrices discussed
above (and possibly others, cf.\ \cite{StasinskiVoll/17}) is that the
relevant rational Poincar\'e series have interpretations in terms of
statistics on Weyl groups of type $A$ and $B$;
cf.\ Remark~\ref{rem:similar}.

\subsection{Notation}
We write $\N = \{1,2,\dots \}$ for the set of natural numbers. For a
subset $X \subseteq \N$ we write $X_0 = X \cup \{0\}$. Given
$m,n\in\Z$, we set $[n] = \{1,\dots,n\}$ and
$[m,n]=\{m,\dots,n\}$. The notation $I = \{i_1,\dots,i_\ell\}_<$ for a
finite subset $I\subseteq \N_0$ indicates that $i_1<\dots<i_\ell$. For
a property~$P$, we write $\delta_P$ for the ``Kronecker delta'' of
$P$, which is $1$ if $P$ holds and $0$ otherwise. We write $\Mat_n(R)$
for the ring of $n\times n$-matrices over a ring $R$. The notation
$\Mat_{n,m}(\Fq)$, however, is reserved for the set of
$n\times n$-matrices over the finite field~$\mathbb{F}_q$ of
rank~$m$. By $R^\times$ we denote the group of units of a
(commutative) ring~$R$. Given \rev{$i\in[n]_0$ and $I\subseteq [n]_0$}
as above, we write
\rev{$$\binom{n}{i}_X = \prod_{j=1}^i
  \frac{1-X^{n-i+j}}{1-X^j}\in\Z[X]$$
for the $X$-binomial coefficient and
$$\binom{n}{I}_X = \binom{n}{i_\ell}_X
\binom{i_{\ell}}{i_{\ell-1}}_X\dots \binom{i_2}{i_1}_X$$
for the $X$-multinomial coefficient.} We use the notation $\gp(X)$ for
the geometric progression
$\sum_{r=1}^\infty X^r =\frac{X}{1-X}\in\Q(X)$ and write $\Tr(\bfx)$
for the trace of a square matrix~$\bfx$.

\section{Counting traceless matrices}

\subsection{Counting traceless matrices over finite fields by rank}
Formulae for the cardinalities $|B_{n,n-i}(\Fq)| = |\{ \bfx \in
\sll_n(\Fq) \mid \rk(\bfx) = n-i\}|$, for $i\in[n-1]_0$, are given, for
example, in \cite{Bender/74}.  We recall them here, suitably
rephrased. Recall the definition
\begin{equation}\label{def:G}
  f_{n,i}(X) := f_{G_{n},\{i\}}(X)  = \binom{n}{i}_{X}\prod_{j=i+1}^n(1-X^j) \in\Z[X]
\end{equation}
from
\cite[Theorem~C]{StasinskiVoll/14}. \cite[Lemma~3.1~(3.3)]{StasinskiVoll/14}
asserts that
$$q^{n^2-i^2}f_{n,i}(q^{-1})$$
is the number $\left| \Mat_{n,n-i}(\Fq)\right|$ of
$n\times n$-matrices over $\Fq$ of rank $n-i$. Set
\begin{align}
  b_{n,i}(X) & = f_{n,i}(X) +
  (-1)^{n-i}\binom{n}{i}_X(1-X)X^{\binom{n+1}{2} - \binom{i+1}{2}
    -1} \label{equ:bni}\\ &=
  \binom{n}{i}_X\left(\left(\prod_{j=i+1}^n(1-X^j)\right) +
  (-1)^{n-i}(1-X)X^{\binom{n+1}{2} - \binom{i+1}{2} -1}\right)
  \in\Z[X].\nonumber
\end{align}

\begin{remark}\label{rem:bni}
  Informally, we obtain $b_{n,i}(X)$ from $f_{n,i}(X)$ by lowering by
  one the exponent in the leading term in the
  factor~$\prod_{j=i+1}^n(1-X^j)$ of $f_{n,i}(X)$.
\end{remark}
\begin{lemma}[\cite{Bender/74}]\label{lem:bender} For $i\in[n-1]_0$,
 $$\left|B_{n,n-i}(\Fq)\right| = q^{n^2-i^2-1}b_{n,i}(q^{-1}).$$
\end{lemma}

\begin{proof}
  In eq.\ (1) of his paper \cite{Bender/74}, Bender states that
\begin{equation}\label{bender}
  \left|B_{n,n-i}(\Fq)\right| = q^{-1}q^{\binom{n-i}{2}}\left(\prod_{j=1}^{n-i}\frac{(q^{n-j+1}-1)^2}{q^j-1}\right) + (1-q^{-1})\prod_{j=1}^{n-i}\frac{q^{j-1}-q^n}{q^j-1}.
\end{equation}
It is a triviality to see that the first summand on the right-hand
side of \eqref{bender} is equal to $q^{n^2-i^2-1}f_{n,i}(q^{-1})$,
whereas the second summand is equal to
\begin{equation}\label{error}
  q^{n^2-i^2-1} (-1)^{n-i}\binom{n}{i}_{q^{-1}}(1-q^{-1})q^{-\binom{n+1}{2} + \binom{i+1}{2}+1}.\qedhere
\end{equation} 
\end{proof}

We note that \eqref{error} gives the ``error term''
$\left|B_{n,n-i}(\Fq)\right| - \left| \Mat_{n,n-i}(\Fq)\right| / q$.

\subsection{Counting traceless matrices over quotients of compact discrete valuation rings by elementary divisor type}\label{subsec:counting.edt}
Our aim is to generalize the enumeration of traceless matrices over
finite fields by their ranks to traceless matrices over larger
quotients of compact discrete valutation rings by their image
sizes. The latter, in turn, are controlled by the matrices' elementary
divisor types. More precisely, given
$I=\set{i_1,\ldots,i_{\ell}}_{<}\subseteq [n-1]_0$ and
$\bfr_I\in\N^I$, we set $\mu_j=i_{j+1}-i_j$ with $i_{\ell+1}=n$ and
$i_0=0$, and $N=\sum_{\iota\in I} r_{\iota}$. Recall, e.g.\ from
\cite[\S3]{StasinskiVoll/14}, that a (primitive) $n\times n$-matrix
$\bfx$ over $\lri_N = \lri/\mfp^N$ is said to be of elementary divisor
type $(I,\bfr_I)$ if
\[\nu(\bfx)=(\underbrace{0,\ldots,0}_{\mu_\ell},\underbrace{r_{i_\ell},\ldots,r_{i_\ell}}_{\mu_{\ell-1}},\underbrace{r_{i_\ell}+r_{i_{\ell-1}},\ldots,r_{i_\ell}+r_{i_{\ell-1}}}_{\mu_{\ell-2}},\ldots,\underbrace{N,\ldots,N}_{\mu_0}),\]
where $\nu(\bfx)$ is the $n$-tuple of valuations of the elementary
divisors of $\bfx$, in nondescending order. In analogy with the
notation (introduced in
\cite[Section~3]{StasinskiVoll/14})
$$\rmN^\lri_{I,\bfr_{I}}(G_n)$$ for the set of $n\times n$-matrices
over $\lri_{N}$ of elementary divisor type $(I,\bfr_I)$, we write
\[\rmN^\lri_{I,\bfr_{I}}(K_n)\]
for the set of traceless $n\times n$-matrices over $\lri_{N}$ of elementary divisor type
$(I,\bfr_I)$.

Let $t=q^{-s}$. Noting that a matrix $\bfx\in \sll_n(\lri_N)$ of
elementary divisor type $(I,\bfr_I)$ has image size
$| \im(\bfx)| = q^{\sum_{i\in I}r_i(n-i)}$, equation
\cite[(3.1)]{StasinskiVoll/14} gives
\begin{equation}\label{equ:poincare}
 \mathcal{P}_{n,\lri}(s) = \sum_{I \subseteq [n-1]_0} \sum_{\bfr_I\in\N^I} |\rmN^\lri_{I,\bfr_{I}}(K_n)|
 t^{\sum_{i\in I}r_{i}(n-i)},
\end{equation}
reducing the problem of computing $\mathcal{P}_{n,\lri}(s)$ to the one
of effectively describing the numbers $|\rmN^\lri_{I,\bfr_{I}}(K_n)|$
for varying $I$ and $\bfr_I$. The following result reduces the latter
problem further to the problems of counting traceless matrices over
the residue field~$\Fq$ and counting smaller, but not necessarily
traceless matrices over a proper quotient of $\lri_N$.
\begin{proposition}\label{pro:Kn-to-Gil}
  Given $\varnothing\neq I\subseteq [n-1]_0$ and $\bfr_I\in\N^I$, with
  $i_\ell = \max I$,
\[|\rmN^\lri_{I,\bfr_{I}}(K_n)| = | B_{n,n-i_\ell}(\Fq)| q^{(N-1)(n^2 - i_\ell^2 -1)} \cdot|\rmN^\lri_{I\setminus\{i_\ell\},\bfr_{I\setminus\{i_\ell\}}}(G_{i_\ell})|.
\]
\end{proposition}

\newcommand{\injvec}{\imath}
\newcommand{\injmat}{\jmath}

Let us fix some notation for the proof of
Proposition~\ref{pro:Kn-to-Gil}.  Given $1\le k,m\le n$ and $N\in\N$
we define the injection
\begin{align*}
  \injvec_k:\mathbb{A}_{n-1}(\lri_N)\to
  \mathbb{A}_{n}(\lri_N), \quad (t_1,\ldots,t_{n-1})\mapsto
  (t_1,\ldots,t_{k-1},1, t_k,\ldots,t_{n-1})
\end{align*} and let
\begin{align*}\injmat_{k,m}:\Mat_{n-1}(\lri_N)\to\Mat_{n}(\lri_N), \quad \bfx
  \mapsto \bfx'
\end{align*}
be the map sending an $(n-1)\times(n-1)$-matrix $\bfx$ to the
$n\times n$-matrix $\bfx'$ whose $k$-th row and $m$-th column are
zero, and such that the submatrix obtained by deleting the $k$-th row
and $m$-th column is $\bfx$.

The proof of Proposition~\ref{pro:Kn-to-Gil} also requires the
following lemma.

\begin{lemma}\label{lem:conj-k,n-inv-coord}
  \rev{ Let $\mathbf{a}=(a_{i,j})\in \sll_n(\Fq)$. Assume that
    $\mathbf{a}$ is noncentral and $a_{n,n}\neq 0$. Then there exist
    $1\le k<n$ and $u\in \SL_n(\Fq)$ such that
  \begin{enumerate}
  \item the $(n,n)$-entry of $u\mathbf{a}u^{-1}$ is nonzero and
  \item the $(n,k)$- and $(k,n)$-entries of $u\mathbf{a}u^{-1}$ are
    not both zero.
  \end{enumerate}
}
\end{lemma}
\begin{proof}
  \rev{ The assumption that $\mathbf{a}$ is not central implies that
    one of the following holds:
\begin{enumerate}
\item \label{case1} all diagonal entries of $\mathbf{a}$ are equal to
  $a_{n,n}$ and $a_{k,r} \neq 0$ for some $r\ne k$ or
\item \label{case2} there exists $1\le k<n$ such that
  $a_{k,k}\ne a_{n,n}$.
\end{enumerate}

In case \eqref{case1}, we may take $u=-\sigma_{r,n}$ where
$\sigma_{r,n}$ is the permutation matrix interchanging the $r$-th and
$n$-th elements of the standard basis and fixing all others. 

In case \eqref{case2}, we may assume that $a_{k,n}=a_{n,k}=0$;
otherwise the lemma holds trivially. In this setting, one may take $u$
to be the elementary matrix with $1$ in position $(k,n)$ and verify
that conjugation of $\mathbf{a}$ by $u$ leaves the $(n,n)$-entry
unchanged and puts the nonzero value $a_{n,n}-a_{k,k}$ in the
$(k,n)$-entry.  }
\end{proof}

\begin{proof}[Proof of Proposition~\ref{pro:Kn-to-Gil}] We denote by
  $$\phi: \rmN^\lri_{I,\bfr_I}(K_n) \to B_{n,n-i_\ell}(\Fq)$$ the
  reduction modulo $\mfp$. We will prove that, for any
  $\mathbf{a}=(a_{i,j})\in B_{n,n-i_\ell}(\Fq)$, the fibre of $\phi$
  over $\mathbf{a}$ has size
\[\abs{\phi^{-1}(\mathbf{a})}=q^{(N-1)(n^2-i_\ell^2-1)}\abs{\rmN^\lri_{I\setminus\set{i_\ell},\bfr_{I\setminus\set{i_\ell}}}(G_{i_\ell})}.\]
This clearly suffices to prove the proposition. Given $1\le i,j\le n$,
we define a map
\begin{align*} \Phi_{i,j}:\lri_N^\times\times
  \mathbb{A}_{n-1}(\lri_{N})\times
  \mathbb{A}_{n-1}(\lri_N)\times\Mat_{n-1}(\lri_N)&\to
  \Mat_n(\lri_N)\\ (x,\mathbf{c},\mathbf{r},\mathbf{y})&\mapsto
  x\left(\injvec_{i}(\mathbf{c})^\tr\injvec_{j}(
    \mathbf{r})+\injmat_{i,j}(\mathbf{y})\right).
 \end{align*}

 Note that $\Phi_{i,j}$ is a bijection onto the set of matrices
 $\bfx\in \Mat_n(\lri_N)$ with invertible $(i,j)$-entry. In
 particular, $\im(\Phi_{i_0,j_0})\supseteq \varphi^{-1}(\mathbf{a})$
 for any $1\le i_0,j_0\le n$ such that $a_{i_0,j_0}\ne 0$.
 
Given $a_{i_0,j_0}\ne 0$, the requirement
$\Phi_{i_0,j_0}(x,\mathbf{c},\mathbf{r},\mathbf{y})\equiv
\mathbf{a}\pmod{\mfp}$
is equivalent to the conjunction of the following conditions:
\begin{enumerate}\renewcommand{\theenumi}{\roman{enumi}}
\item\label{item:xcond}  $x\equiv a_{i_0,j_0}\pmod\mfp$,
\item\label{item:ccond} $\injvec_{i_0}(\mathbf{c})^\tr$ is congruent
  modulo $\mfp$ to the $j_0$-th column of the matrix
  $a_{i_0,j_0}^{-1}\cdot \mathbf{a}$,
\item\label{item:rcond} $\injvec_{j_0}(\mathbf{r})$ is congruent
  modulo $\mfp$ to the $i_0$-th row of the matrix
  $a_{i_0,j_0}^{-1}\cdot \mathbf{a}$, and
\item\label{item:ycond} the matrix of $\injmat_{i_0,j_0}(\mathbf{y})$
  is congruent modulo $\mfp$ to the matrix $\mathbf{b}=(b_{i,j})$,
  where $b_{i,j}=
  \left(\frac{a_{i,j}}{a_{i_0,j_0}}\right)-\left(\frac{a_{i,j_0}}{a_{i_0,j_0}}\right)\left(\frac{a_{i_0,j}}{a_{i_0,j_0}}\right)$
  for all $1\le i,j\le n$.
\end{enumerate}
In particular, since the elementary divisor
type of a matrix is invariant under elementary row and column
operations, for any choice of $x\in\lri_N^\times$ and $\mathbf{c} =
(c_1,\ldots,c_{n-1}), \mathbf{r} = (r_1,\ldots,r_{n-1})
\in\mathbb{A}_{n-1}(\lri_N)$ satisfying conditions
\eqref{item:xcond},\eqref{item:ccond}, and \eqref{item:rcond} above,
there exists a matrix $\mathbf{y}\in\Mat_{n-1}(\lri_N)$, satisfying
condition \eqref{item:ycond} such that
$\Phi_{i_0,j_0}(x,\mathbf{c},\mathbf{r},\mathbf{y})$ is of elementary
divisor type $(I,\bfr_I)$. We now compute the number of such tuples
which also satisfy the condition of being traceless. We proceed by a
case distinction depending on the diagonal entries of $\mathbf{a}$.

\begin{list}{}{\setlength{\leftmargin}{5pt}
   \setlength{\itemsep}{2pt} \setlength{\parsep}{1pt}}
\item \textit{Case 1: $\mathbf{a}$ is noncentral and has a nonzero
    diagonal entry.} In this case, since $\abs{\phi^{-1}(\mathbf{a})}$
  is invariant under conjugating $\mathbf{a}$ by an element of
  $\SL_n(\Fq)$ we may, for simplicity, assume that~$i_0=j_0=n $. By
  the same token, Lemma~\ref{lem:conj-k,n-inv-coord}, \rev{and
    invariance of $\abs{\phi^{-1}(\mathbf{a})}$ under transposition of
    the matrix~$\mathbf{a}$}, we may assume that there exists
  $1\le k_0\le n-1$ such that $a_{k_0,n}\neq 0$ as well. In
  particular, it follows that the $k_0$-th entry $c_{k_0}$ of any
  element $\mathbf{c} \in\mathbb{A}_{n-1}(\lri_N)$ satisfying
  condition \eqref{item:ccond} is invertible in~$\lri_N$.

  A direct computation yields
\begin{align*}
  \Tr(\Phi_{n,n}(x,\mathbf{c},\mathbf{r},\mathbf{y}))&=x
                                                       \left(\injvec_{n}(\mathbf{r})\injvec_{n}(\mathbf{c})^\textup{tr}+\Tr(\injmat_{n,n}(\mathbf{y}))\right)=x \left(\sum_{k=1}^{n-1}c_k
                                                       r_k +1+\Tr(\mathbf{y})\right).
\end{align*}
As $x\in\lri_N^\times$, the requirement
$\Tr(\Phi_{n,n}(x,\mathbf{c},\mathbf{r},\mathbf{y}))=0$ is equivalent
to
\[\sum_{k=1}^{n-1}c_kr_k=-1-\Tr(\mathbf{y}).\]
Recalling that $c_{k_0}\in\lri_N^{\times}$, one easily verifies that
any choice of $\mathbf{y}\in\Mat_{n-1}(\lri_N)$ which satisfies
condition \eqref{item:ycond} admits exactly $q^{(N-1)(2n-2)}$ triples
$(x,\mathbf{c},\mathbf{r})\in \lri_N^\times\times
\mathbb{A}_{n-1}(\lri_N)\times \mathbb{A}_{n-1}(\lri_N)$
such that $\Phi_{n,n}(x,\mathbf{c},\mathbf{r},\mathbf{y})$ is traceless and
reduces modulo $\mfp$ to~$\mathbf{a}$. Furthermore, the elementary
divisor type of $\Phi_{n,n}(x,\mathbf{c},\mathbf{r},\mathbf{y})$ is
determined by that of $\bfy$ as follows.
\begin{enumerate}
\item Suppose that $\mathbf{a}\in\Mat_{n,1}(\Fq)$,
  i.e.\ $\rk(\mathbf{a})=1$ and $i_\ell=\max I=n-1$. Let
  $\varpi\in\mfp$ be a uniformizer.  The matrix $\mathbf{b}$ in
  condition \eqref{item:ycond} is zero, \rev{so} 
  $\Phi_{n,n}(x,\mathbf{c},\mathbf{r},\mathbf{y})$ is of elementary
  divisor type $(I,\bfr_I)$ if and only if
  $\mathbf{y}\in\mfp^{r_{i_\ell}}\Mat_{n-1}(\lri_N)$ is such that
  $\varpi^{-r_{i\ell}}\mathbf{y}$ is of elementary divisor type
  $(I\setminus\set{i_\ell},\bfr_{I\setminus\set{i_\ell}})$. Thus
\begin{align*}
\abs{\phi^{-1}(\mathbf{a})}&=q^{(N-1)(2n-2)}\abs{\rmN_{I\setminus\set{i_\ell},\bfr_{I\setminus\set{i_\ell}}}^\lri
  (G_{n-1})}\\&=q^{(N-1)(n^2-i_\ell^2-1)}\abs{\rmN^\lri_{I\setminus\set{i_\ell},\bfr_{I\setminus\set{i_\ell}}}(G_{i_\ell})
}.\end{align*}
\item Otherwise, if $i_\ell=\max I<n-1$, the matrix $\mathbf{b}$ is
  nonzero, and $\Phi_{n,n}(x,\mathbf{c},\mathbf{r},\mathbf{y})$ is of
  elementary divisor type $(I,\bfr_I)$ if and only if $\mathbf{y}$ is
  an $(n-1)\times (n-1)$-matrix of elementary divisor type
  $(I,\bfr_I)$. Arguing as in \cite[Proposition~3.4]{SV1/15}, the
  number of such lifts $\mathbf{y}$ of $\mathbf{b}$ is
  $q^{(N-1)((n-1)^2-i_\ell^2)}|\rmN_{I\setminus\set{i_\ell},\bfr_{I\setminus\set{i_\ell}}}^\lri(G_{i_\ell})|$. Combined
  with the $q^{(N-1)(2n-2)}$ possibilities to choose
  $(x,\mathbf{c},\mathbf{r})$ for any such $\mathbf{y}$, we obtain
\[\abs{\phi^{-1}(\mathbf{a})}=q^{(N-1)(n^2-i_\ell^2-1)}\abs{\rmN_{I\setminus\set{i_\ell},\bfr_{I\setminus\set{i_\ell}}}^{\lri}(G_{i_\ell})},\]
as wanted.
\end{enumerate}
\item \textit{Case 2: all diagonal entries of $\mathbf{a}$ are zero.}
  In this case $i_0\ne j_0$ and
\begin{align*}
  \Tr(\Phi_{i_0,j_0}(x,\mathbf{c},\mathbf{r},\mathbf{y}))&=x(\injvec_{j_0}(\mathbf{r}) \injvec_{i_0}(\mathbf{c})^{\textup{tr}}+\Tr(\injmat_{i_0,j_0}(\mathbf{y})))\\
                                               &=x(c_{j_0}+r_{i_0}+\theta_{i_0,j_0}(\mathbf{r},\mathbf{c})+\tau_{i_0,j_0}(\mathbf{y}))
\end{align*}
where $\theta_{i_0,j_0}(\mathbf{r},\mathbf{c})$ is some quadratic
function in
$\mathbf{r}=(r_1,\ldots,r_{n-1}),\mathbf{c}=(c_1,\ldots,c_{n-1})$
which does not involve $r_{i_0}$ or $c_{j_0}$, and
$\tau_{i_0,j_0}(\mathbf{y})$ is some linear function in the entries
of~$\mathbf{y}$. In particular, the condition
$\Tr(\Phi_{i_0,j_0}(x,\mathbf{c},\mathbf{r},\mathbf{y}))=0$ is
equivalent to
\[c_{j_0}=-r_{i_0}-\theta_{i_0,j_0}(\mathbf{r},\mathbf{c})-\tau_{i_0,j_0}(\mathbf{y}).\]
Consequently, any choice of $\mathbf{y}$ which satisfies condition
\eqref{item:ycond} admits exactly $q^{(N-1)(2n-2)}$ triples
$(x,\mathbf{c},\mathbf{r})\in\lri^\times_N\times\mathbb{A}_{n-1}(\lri_{N})\times\mathbb{A}_{n-1}(\lri_N)$
such that $\Phi_{i_0,j_0}(x,\mathbf{c},\mathbf{r},\mathbf{y})$ is
traceless and reduces to $\mathbf{a}$ modulo~$\mfp$. Considering the
same two possibilities for $\mathbf{a}$ as in the previous case and
applying the same arguments, we obtain
\[\abs{\phi^{-1}(\mathbf{a})}=q^{(N-1)(n^2-i_\ell^2-1)}\abs{\rmN^\lri_{I\setminus\set{i_\ell},\bfr_{I\setminus\set{i_\ell}}}(G_{i_\ell})}\]
in this case as well.

\item \textit{Case 3: $\mathbf{a}$ is central.} Note that this case is
  only possible if $p$ divides $n$ and $I=\set{0}$. Thus, our goal in
  this case is to prove
\[\abs{\varphi^{-1}(\mathbf{a})}=q^{(N-1)(n^2-1)}.\]
Since the lift of an element of $\Mat_{n,n}(\Fq)$ (viz.\ an invertible
matrix over $\Fq$) to $\Mat_n(\lri_N)$ is always an element of
elementary divisor type $(I,\bfr_I)=(\set{0},(N))$ (viz.\ an
invertible matrix over $\lri_N$), we simply need to note that such an
element $\mathbf{a}$ admits $q^{(N-1)(n^2-1)}$ traceless lifts to
$\Mat_n(\lri_N)$. 
\end{list}
This concludes the proof of
Proposition~\ref{pro:Kn-to-Gil}.
\end{proof}

Recall from \cite[Theorem~C]{StasinskiVoll/14} the definition, for
$I\subseteq [n-1]_0$, of the polynomials
\begin{equation}\label{def:fnI}
  f_{n,I}(X) := f_{G_{{n},I}}(X) = \binom{n}{I}_X \prod_{j=\min I
    +1}^n(1-X^j) \in\Z[X],
\end{equation}
generalizing those defined in \eqref{def:G} for singletons $I= \{i\}$.

\begin{proposition}\label{pro:Kn}
  Given $\varnothing \neq I\subseteq [n-1]_0$ and $\bfr_I\in\N^I$,
  with $i_\ell = \max I$,
$$|\rmN^\lri_{I,\bfr_{I}}(K_n)| = b_{n,i_{\ell}}(q^{-1})
  f_{i_{\ell}, I \setminus \{i_\ell\}}(q^{-1}) q^{\sum_{\iota\in I}
    r_\iota(n^2-\iota^2-1)}.$$
\end{proposition}

\begin{proof} Recall that $N = \sum_{\iota\in I}r_\iota$.  Using
  \cite[Proposition~3.4~(3.8)]{StasinskiVoll/14}, which asserts that
$$|\rmN^\lri_{I\setminus\{i_\ell\},\bfr_{I\setminus\{i_\ell\}}}(G_{i_\ell})|
  = f_{i_{\ell}, I \setminus \{i_\ell\}}(q^{-1}) q^{\sum_{\iota\in
      I\setminus\{i_\ell\}} r_\iota(i_\ell^2-\iota^2)},$$ and
  Lemma~\ref{lem:bender}, Proposition~\ref{pro:Kn-to-Gil} yields
\begin{multline*}
  |\rmN^\lri_{I,\bfr_{I}}(K_n)| = | B_{n,n-i_\ell}(\Fq)| q^{(N-1)(n^2
    - i_\ell^2 -1)}
  \cdot|\rmN^\lri_{I\setminus\{i_\ell\},\bfr_{I\setminus\{i_\ell\}}}(G_{i_\ell})|=\\ |
  B_{n,n-i_\ell}(\Fq)| q^{(r_{i_\ell}-1)(n^2 - i_\ell^2
    -1)}q^{\left(\sum_{\iota \in I \setminus
      \{i_\ell\}}r_\iota\right)(n^2 - i_\ell^2 -1)} \cdot f_{i_{\ell}, I
    \setminus \{i_\ell\}}(q^{-1}) q^{\sum_{\iota\in
      I\setminus\{i_\ell\}}
    r_\iota(i_\ell^2-\iota^2)}=\\ b_{n,i_\ell}(q^{-1}) f_{i_{\ell}, I
    \setminus \{i_\ell\}}(q^{-1}) q^{\sum_{\iota\in
      I}r_\iota(n^2-\iota^2-1)}.\qedhere
\end{multline*}
\end{proof}

We apply Proposition~\ref{pro:Kn} to obtain a first formula for the
Poincar\'e series $\mcP_{n,\lri}(s)$.  Recall that $\gp(X) =
\frac{X}{1-X}$ and $t = q^{-s}$.

\begin{theorem}\label{thm:Pn.first} Setting
  $x_{n,i} = q^{n^2-i^2-1}t^{n-i}$ for $i\in[n-1]_0$, the image zeta
  function of $\sll_n(\lri)$ satisfies
$$\mathcal{P}_{n,\lri}(s) = 1 + \sum_{i=0}^{n-1} b_{n,i}(q^{-1})
  \gp(x_{n,i}) \left(\sum_{J\subseteq [i-1]_0} f_{i,J}(q^{-1})
  \prod_{j\in J}\gp(x_{n,j})\right).$$
\end{theorem}

\begin{proof} Rewriting \eqref{equ:poincare} yields
\begin{align*}
\mathcal{P}_{n,\lri}(s) &= \sum_{I \subseteq [n-1]_0} \sum_{\bfr_I\in\N^I} |\rmN^\lri_{I,\bfr_{I}}(K_n)|
 t^{\sum_{i\in I}r_{i}(n-i)} \\
&= 1 + \sum_{i_\ell=0}^{n-1} \sum_{\substack{I \subseteq [n-1]_0\\ \max I = i_{\ell}}} \sum_{\bfr_I\in\N^I} |\rmN^\lri_{I,\bfr_{I}}(K_n)|
 t^{\sum_{i\in I}r_{i}(n-i)}.
\end{align*}
For fixed $i_\ell\in[n-1]_0$ we find, using Proposition~\ref{pro:Kn},
that
\begin{multline*}
  \sum_{\substack{I \subseteq [n-1]_0\\ \max I = i_{\ell}}}
  \sum_{\bfr_I\in\N^I} |\rmN^\lri_{I,\bfr_{I}}(K_n)| t^{\sum_{i\in
      I}r_{i}(n-i)}
  =\\ b_{n,i_\ell}(q^{-1})\gp(x_{n,i_\ell})\sum_{J\subseteq[i_\ell-1]_0}f_{i_\ell,J}(q^{-1})
  \prod_{j\in J} \gp(x_{n,j}).\qedhere
\end{multline*}

\end{proof}

\section{Traceless matrices and signed permutation statistics}

In this section we express the Poincar\'e series $\mcP_{n,\lri}(s)$ in
terms of certain signed permutation statistics
(Proposition~\ref{pro:Pn.second}) and Igusa functions
(Theorem~\ref{thm:Pn.factor}).

\subsection{Preliminaries on signed permutation groups}\label{subsec:coxeter.prelim}
We collect a few definitions and notation regarding (signed)
permutation groups, mostly standard and covered in general references
such as~\cite{BjoernerBrenti/05}.

We write $S_n$ for the symmetric group of degree $n$, viz.\ the group
of permutations of the set $[n]$. It is a Coxeter group with Coxeter
generating set $\{s_1,\dots,s_{n-1}\}$, consisting of the standard
transpositions $s_i$ interchanging letters $i$ and~$i+1$.

We denote further by $B_n$ the group of signed permutations of degree
$n$, viz.\ permutations $w$ of $[-n,n]$ satisfying $w(-i)=-w(i)$ for
all $i\in[n]_0$. Signed permutations are uniquely determined by their
restrictions to $[n]$. This is exploited in the one-line notation,
representing $w\in B_n$ by its values $w(1)\, w(2)\, \dots
\,w(n)$.
When using the one-line notation we write $\bar a$ instead of $-a$
for~$a\in\Z$.

The group $B_n$ is a Coxeter group with Coxeter generating set
$S^{B_n} = \{s_0,s_1,\ldots s_{n-1}\}$, where
$s_0=\bar 1 2 3 \ldots n$ and, for $i\in [n-1]$,
$s_i=1 2 \ldots (i +1) i\ldots n$ are the standard transpositions. Let
$\ell$ denote the Coxeter length with respect to $S^{B_n}$.

We identify $S_n$ with the (parabolic) subgroup of $B_n$ consisting of
elements $w$ satisfying $w(i)>0$ for all $i\in[n]$, generated by the
transpositions $s_i$ for~$i\in[n-1]$. The restriction of the Coxeter
length $\ell$ on $B_n$ to $S_n$ coincides with the Coxeter length on
$S_n$, so the use of the \rev{notation} '$\ell$' is unambigious.

Let $w\in B_n$. The \emph{negative set} of $w$ is
$$\Neg(w)=\{i\in [n]\mid w(i)<0\}$$ 
and the \emph{descent set} of $w$ is 
$$\Des(w) = \{ i \in [n-1]_0 \mid \ell(w s_i) < \ell(w) \}.$$
Recall also the following statistics on $B_n$:
\begin{alignat*}{2}
  \rmaj(w) & = \sum_{i \in \Des(w)} (n-i)&&(\textup{reverse major index})\\
  \des(w) & = |\Des(w)|&&(\textup{descent number})\\ \inv(w) & =
  |\{(i,j)\in [n]^2 \mid i<j,\,w(i)>w(j)\}|&&(\textup{inversion
    number})\\ \nega(w) &= |\Neg(w)|&&(\textup{negative number})\\
  \nsp(w) &= |\{(i,j)\in [n]^2 \mid i<j, \,
  w(i)+w(j)<0\}|\quad&&(\textup{negative sum pair
    number})\end{alignat*}

It is well-known (cf.\ \cite[Proposition~8.1.1]{BjoernerBrenti/05})
that the Coxeter length $\ell$ on $B_n$ satisfies
\begin{equation}\label{equ:length.Bn}
 \ell=\ \inv+\nega+\nsp.
\end{equation}

We further set
\begin{alignat*}{2}
\sigma_B(w) &= \sum_{i\in \Des(w)} (n^2 - i^2)&\quad& \textup{ for $w\in B_n$ and}\\
\sigma_A(w) &= \sum_{i\in \Des(w)} i(n-i)&\quad& \textup{ for $w\in S_n$}.
\end{alignat*}
The statistics $\sigma_A$ on $S_n$ and $\sigma_B$ on $B_n$ have a
uniform definition in terms of simple root coefficients of the sum
over all (positive) roots in the pertinent root systems;
cf.\ \cite[p.~510]{StasinskiVoll/14}. The statistics $\sigma_A - \ell$
and $\sigma_B-\ell$ are examples of Weyl group statistics that ``ought
to be better known'', as the authors of \cite{StembridgeWaugh/98}
argue. In any case, they both feature in one of this section's main
results, viz.\ Proposition~\ref{pro:factor}.

In \cite[Section 1.2.1]{AdinBrentiRoichman/05}, the statistic $\nch$
on $B_n$ is defined by setting, for $w\in B_n$,
\begin{equation}\label{def:epsn}\nch(w)= \rev{\delta_{w(n)<0}} =\begin{cases} 1 & \text{if } w(\rev{j})<0 \text{ for all } \rev{j} \in [\max \Des(w)+1,n], \\
    0 & \text{otherwise}.
  \end{cases}
\end{equation}

For $I\subseteq [n-1]_0$, the corresponding \emph{quotient}
of $B_n$ is defined as
$$B_n^{I^c} = \{ w\in B_n \mid \Des(w) \subseteq I\}.$$ Quotients of
$S_n$ are defined analogously.
                                                                  
\subsection{Signed permutation statistics and enumeration of matrices}
                                 
It is known by \cite[Proposition~4.6]{StasinskiVoll/14}---essentially
a result of Reiner's; cf.\ \cite{Reiner-Signed-perm}---that the
polynomials $f_{n,I}(X)$ defined in \eqref{def:fnI} satisfy
\begin{equation}\label{equ:fni}
f_{n,I}(X) = \sum_{w\in B_n^{I^{\textup{c}}}}(-1)^{\nega(w)}X^{\ell(w)}.
\end{equation}
In particular, by \cite[Lemma~3.1~(3.3)]{StasinskiVoll/14}, for
$i\in [n-1]_0$, the number $|\Mat_{n,n-i}(\Fq)|$ of
$n\times n$-matrices over $\Fq$ of rank $n-i$ satisfies
\begin{equation*}
  |\Mat_{n,n-i}(\Fq)| = q^{n^2-i^2}\sum_{w\in B_n^{\{i\}^c}}
  (-1)^{\nega(w)}q^{-\ell(w)}.
\end{equation*}
Key to a similar interpretation of the numbers $|B_{n,n-i}(\Fq)|$ in
terms of signed permutation statistics in
Proposition~\ref{pro:bender.coxeter} is the following lemma.  

\begin{lemma}\label{lem:bender.cox}
  For $i\in [n-1]_0$, the polynomial $b_{n,i}(X)$ defined in
  \eqref{equ:bni} satisfies
$$b_{n,i}(X) = \sum_{w\in B_n^{\{i\}^c}} (-1)^{\nega(w)}X^{(\ell - \nch)(w)}.$$
\end{lemma}
\begin{proof}
  Thanks to Remark~\ref{rem:bni} it suffices to prove that
  \begin{equation}\label{eq:lenlong}
    \sum_{\substack{\{w\in B_n^{\{i\}^{\textup{c}}}\mid
        w(n)<0\}}}(-1)^{\nega(w)}X^{\ell(w)}= \binom{n}{i}_X(-1)^{n-i}
    X^{\binom{n+1}{2}-\binom{i+1}{2}}.\end{equation} To
  prove \eqref{eq:lenlong} we first note that, for $w\in
  B_{n}^{\{i\}^c}$, the condition $w(n)<0$ is equivalent to $w(j)<0$
  for $j=[i+1,n]$; cf.~\eqref{def:epsn}. Clearly $\nega(w) = n-i$ in this case. The set $\{w\in
  B_n^{\{i\}^{\textup{c}}}\mid w(n)<0\}$ is in bijection with
  $S_{n}^{\{n-i\}^c}$ through the map $w\mapsto \widetilde{w} :=
  \overline{w(n)}\dots\overline{w(i+1)}w(1)\dots w(i)$. In light of
  the well-known identity $$\sum_{\widetilde{w}\in
    S_n^{\{n-i\}^{\textup{c}}}}X^{\ell(\widetilde{w})}=
  \binom{n}{n-i}_X=\binom{n}{i}_X$$ it suffices to prove that
  \begin{equation}
	\label{eq:diffell}
\ell( w)-\ell(\widetilde w)=\binom{n+1}{2}-\binom{i+1}{2} \mbox{ for
  all } w\in B_n^{\{i\}^{\textup{c}}} \mbox{ with }
w(n)<0.
\end{equation}
By \eqref{equ:length.Bn}, the quantity $\ell(w)$ is the sum of the
quantities
$$\invv(w)=i(n-i),\qquad \nega(w)=n-i,\qquad
\nsp(w)=\binom{n-i}{2}+r(w),$$ where $r(w) := |\{(\sigma,\tau)\in
    [i]\times[i+1,n]\mid w(\sigma)+w(\tau)<0\}|$.  Clearly
    $\ell(\widetilde w)=\invv(\widetilde
    w)=r(w)$. Eq.\ \eqref{eq:diffell} follows, as
    $i(n-i)+(n-i)+\binom{n-i}{2}+r(w)-r(w)=\binom{n+1}{2} -
    \binom{i+1}{2}$.
\end{proof}

\begin{remark}\label{rem:similar}
  Formulae similar to the one given in
  Proposition~\ref{pro:bender.coxeter} hold for the numbers of
  antisymmetric resp.\ symmetric matrices over finite fields of fixed
  ranks. Indeed, let $\delta\in\{0,1\}$ and $i\in[n-1]_0$. 
\begin{enumerate}
\item In the notation of \cite[Section~3]{StasinskiVoll/14}, the
  number $|\Alt_{2n+\delta,2(n-i)}(\Fq)|$ of antisymmetric
  $(2n+\delta)\times(2n+\delta)$-matrices over $\Fq$ of rank $2(n-i)$
  satisfies
$$|\Alt_{2n+\delta,2(n-i)}(\Fq)| =
q^{\binom{2n+\delta}{2}-\binom{2i+\delta}{2}}\sum_{w\in B_n^{\{i\}^c}}
(-1)^{\nega(w)}q^{-(2\ell+(2\delta-1)\nega)(w)};$$
cf.\ \cite[Lemma~3.1~(3.2) and Proposition~4.6]{StasinskiVoll/14}.
\item Likewise, the numbers $|\Sym_{n,n-i}(\Fq)|$ of symmetric
  $n\times n$-matrices over $\Fq$ of rank $n-i$ are given in terms of
  the \emph{odd length} statistic $L$ on $B_n$ defined, for
  $w \in B_n$, by
  \begin{equation}\label{def:odd.length}
    L(w)=\frac{1}{2}|\{(i,j)\in [-n,n]^2\mid
    i<j,\,w(i)>w(j),\,i\not\equiv j \pmod 2\}|,
  \end{equation}
  (cf.\ \cite[Eq.~(1.14)]{StasinskiVoll/14}): combining
  \cite[Lemma~3.1~(3.4)]{StasinskiVoll/14} with
  \cite[Theorem~5.4]{BrentiCarnevale/17} yields
\begin{equation*}
  |\Sym_{n,n-i}(\Fq)|
  =q^{\binom{n+1}{2}-\binom{i+1}{2}}\sum_{w\in B_n^{\{i\}^c}} (-1)^{\ell(w)}q^{-L(w)}.
\end{equation*}
\end{enumerate}
\end{remark}

For $I \subseteq [n-1]_0$, let
\begin{equation}\label{equ:fI}
f_{K_n,I}(X) =\sum_{w\in
  B_n^{I^c}}(-1)^{\nega(w)}X^{(\ell-\nch)(w)}.
\end{equation}
It is easy to see that
\begin{equation}\label{equ:fact}
  f_{K_n,I}(X) = b_{n,\max I}(X) f_{\max I, I \setminus \{\max I \}}(X).
\end{equation}

The following is analogous to
\cite[Proposition~3.4]{StasinskiVoll/14}.

\begin{proposition}
  Let $I\subseteq [n-1]_0$ and $\bfr_I\in\N^I$. Then
$$ |\rmN^\lri_{I,\bfr_{I}}(K_n)| = f_{K_n,I}(q^{-1})q^{\sum_{i\in I}r_i(n^2-i^2-1)}. $$
\end{proposition}

\begin{proof} This follows by combining Proposition~\ref{pro:Kn} and \eqref{equ:fact}.
\end{proof}

\subsection{A joint distribution result for signed permutation statistics}

In our further analysis of the Poincar\'e series $\mcP_{n,\lri}(s)$ we
will make use of the following formal identity.
\begin{proposition}\label{pro:factor} In 
  $\Q[W,X,Y,Z]$, the following identity holds:
  \begin{multline}
    \sum_{w\in B_n}W^{	(\des-\nch )(w)}X^{(\sigma_B -\ell)(w) }Y^{\nega(w)}Z^{\rmaj(w)} = \\\left(\sum_{w\in
        S_n}W^{\des(w)}X^{(\sigma_A-\ell)(w)}(X^n
      Z)^{\rmaj(w)}\right)\prod_{j=0}^{n-1} (1+X^j Y
    Z). \label{equ:factor}\end{multline}
\end{proposition}
To prove Proposition~\ref{pro:factor} we decompose $B_n$ as the
disjoint union of subsets of signed permutations of fixed
\emph{inverse negative set}, i.e.\ the negative set of the inverse
element. Note that, for $w\in B_n$, the set $\Neg(w^{-1})$ simply
encodes, with inverted sign, the negative entries in the one-line
notation for $w$. For $w=1\bar 3 \bar 4 2\in B_4$, for example,
$\Neg(w^{-1})= \Neg(14\bar 2 \bar 3)=\{3,4\}$. Trivially,
$$B_n = \bigcup_{J\subseteq [n]} \{ w\in B_n \mid \Neg(w^{-1}) = J\}.$$

For the proof of Proposition~\ref{pro:factor} we shall need the
following lemma.

\begin{lemma}\label{lem:Jj}
  For $J\subseteq [n]$ and $\max J< j\leq n$,
\begin{multline*} 
	\sum_{\{w\in B_n\mid
      \Neg(w^{-1})=J\cup\{j\}\}}
  W^{	(\des-\nch )(w)}X^{(\sigma_B -\ell)(w) }Y^{\nega(w)}Z^{\rmaj(w)}\\ =X^{n-j}YZ
  \sum_{\{w\in B_n\mid \Neg(w^{-1})=J\}}
 W^{(\des-\nch )(w)} X^{(\sigma_B -\ell)(w) }Y^{\nega(w)}Z^{\rmaj(w)}.\end{multline*}
\end{lemma}

\begin{proof}
  Let $J$ and $j$ be as in the Lemma.  We define a bijective map
$$ \bar{\phantom{w}^j}:\{ w\in B_n \mid \Neg(w^{-1}) = J\} \to\{ w\in B_n \mid \Neg(w^{-1}) = J\cup\{j\}\}, \quad w \mapsto \bar{w}^j$$ 
and control its effect on the relevant statistics. Write
$[n]\setminus J = \{a_1,\ldots,a_s\}_<$, note that $a_s=n$, and set
$\{b_1,\ldots,b_{s-1}\}_< = \{a_1,\ldots,a_s\}\setminus \{j\}$. Given
$w\in B_n$ with $\Neg(w^{-1})=J$, the signed permutation $\bar{w}^j$
is obtained by replacing, in the one-line notation for $w$, the
letters $a_1,\ldots,a_{s-1}$ with $b_1,\ldots,b_{s-1}$ and $a_s=n$
with $\bj$. Informally speaking, the signed permutation matrix
associated with $\bar{w}^j$ is obtained by cyclically permuting the
last $n-(j-1)$ rows of the signed permutation matrix associated with
$w$ and then switching the sign in the $j$-th row. That the map
$\bar{\phantom{w}}^j$ is well-defined and bijective is clear.

Let $k = w^{-1}(n)$, which is to say that the letter $\bj$ appears in
$\bar w ^j$ in position $k$. We claim that
  \begin{align}\label{eq:newl}
    \ell(\bar w ^j)&=\ell(w)+2k-n-1+j, \\
    \label{eq:newdes}
    \Des(\bar w^j) &=\left(\Des(w)\setminus\{k\}\right)\cup\{k-1\}. 
  \end{align}
  These claims, which are proved below, suffice to prove the
  lemma. Indeed, observe that \eqref{eq:newdes} implies that
$ \sigma_B(\bar w
^j) = \sum_{i\in \Des(\bar w^j)} (n^2 - i^2) = \left( \sum_{i\in
    \Des(w)}(n^2 - i^2)\right) -(n^2-k^2) + (n^2-(k-1)^2) =\sigma_B(w)+2k-1$,
whence it follows, with \eqref{eq:newl}, that
$$(\sigma_{B}-\ell)(\bar{w}^j)=  (\sigma_{B}-\ell)(w) + n-j.$$
Eq.~\eqref{eq:newdes} also implies that
$\des(\bar{w}^j)=\des(w)+\delta_{\{k=n\}}$, because
$\des(w) = \des(\bar w^j)$ unless $\bj$ appears in $\bar w^j$ in
position~$n$. Also, $\nch(\bar w^j)=\nch(w)+\delta_{\{k=n\}}$, whence
$$(\des-\nch)(\bar w^j) = (\des-\nch)(w).$$
Clearly $$\nega(\bar w ^j) = \nega(w) + 1$$ and, again by \eqref{eq:newdes}, $$\rmaj(\bar{w}^j) = \rmaj(w) + 1.$$

To prove claim \eqref{eq:newl} recall the identity
$\ell = \inv + \nega + \nsp$ (cf.~\eqref{equ:length.Bn}) and observe
that inversions and negative sum pairs---enumerated respectively by
the statistics $\inv$ and $\nsp$---involving positions $(r,t)$ with
$k < t \leq n$ are the same for $w$ and~$\bar w ^j$. As $w(k)=n$,
there are $n-k$ inversions in $w$ indexed by the pairs $(k,t)$ with
$t \in [k+1,n]$; they all disappear in $\bar w ^j$.  Likewise, as
$\bar w^j(k)=\bj$ (which, by hypothesis, is the minimum of the entries
of $\bar w ^j$ in one-line notation), the pairs $(i,k)$ for
$i\in [k-1]$ are inversions for $\bar w ^j$ but not for $w$. Moreover,
the (negative) entry $\bj$ in $\bar{w}^j$ gives rise to $j-1$
additional negative sum pairs compared with $w$. Hence
  \begin{align*}
    \ell(\bar w ^j)&=(\invv+\nega + \nsp)(\bar w
    ^j)\\&=\invv(w)-(n-k)+(k-1) + \nega(w) + 1+ \nsp(w) + (j-1)\\&=\ell(w)+2k-n -1 +j.
\end{align*}
The claim \eqref{eq:newdes} follows easily from the fact that the
operation $w \mapsto \bar{w}^j$ preserves all relative positions of
the letters of $w$ other than $n$ and $j$. In position $k$, it
replaces the letter $n$ by $\bj$. If $k<n$, then $k\in\Des(w)$ but
$k-1\not\in \Des(w)$ and the operation shifts this descent of $w$ from
position $k$ to a descent of $\bar{w}^j$ in position $k-1$; if $k=n$,
then it produces a descent at $n-1$.
\end{proof}

\begin{proof}[Proof of Proposition~\ref{pro:factor}]
  The  claim follows by observing that, for $w\in S_n = \{ w\in
  B_n \mid \Neg(w^{-1}) = \varnothing\}$,
  $$(\sigma_A + n \rmaj)(w) = \sum_{i\in \Des(w)}\left( i(n-i) +
  n(n-i)\right) = \sigma_B(w)$$ and repeated applications of
  Lemma~\ref{lem:Jj}.
\end{proof}

\begin{remark}
  Setting $W=1$ in \eqref{equ:factor} yields a polynomial which
  factors further:
  \begin{align*}
    \sum_{w\in B_n}X^{(\sigma_B -\ell)(w) }Y^{\nega(w)}Z^{\rmaj(w)} &=
    \left(\sum_{w\in S_n}X^{(\sigma_A-\ell)(w)}(X^n
    Z)^{\rmaj(w)}\right)\prod_{j=0}^{n-1} (1+X^j Y Z)\\ &=
    \prod_{j=0}^{n-1} \left(
    \frac{1-(X^{n+j}Z)^{n-j}}{1-X^{n+j}Z}\right) \left( 1 +
    X^jYZ\right);\end{align*} cf.\ \cite[Propositions~1.7 and
    4.8]{StasinskiVoll/14}. In the setup of
  Proposition~\ref{pro:factor}, however, further factorization of the
  sum over $S_n$ is not to be expected in general. For $n=3$, for
  instance, \rev{computations with {\sf SageMath} show that the polynomial
$$\sum_{w\in S_3}W^{\des(w)}X^{(\sigma_A-\ell)(w)}(X^3 Z)^{\rmaj(w)}
= 1+WX^3Z (1 + X) (1 + X^3Z)+W^2X^{10}Z^3$$
is irreducible.} It factors after the specific substitutions
($(q^{-1},q,-1,q^{-s})$ for $(W,X,Y,Z)$) we perform in our
applications to counting traceless matrices; cf.\ the proof of
Theorem~\ref{equ:Pn.factor}. For $n>3$, however, there appears to be
no ``systematic'' factorization; cf.\ Example~\ref{exa:small.n}.
\end{remark}

We record here a conjectural factorization for a twisted joint
distribution on $B_n$ of several statistics, including one involving
the odd length function $L$ defined in~\eqref{def:odd.length}.

\begin{conjecture}\label{con:Hn1}
  \begin{multline*}
    \sum_{w\in B_n}(-1)^{\ell(w)}X^{\left(\frac{\sigma_B
          +\rmaj}{2}-L\right)(w)}Z^{\rmaj(w)}= \\\left(\sum_{w\in
        S_n}X^{\left(\frac{\sigma_B+\rmaj}{2}-\ell\right)(w)}Z^{\rmaj(w)}\right)\prod_{i=0}^{n-1}(1-X^iZ).\end{multline*}
\end{conjecture}
\begin{remark}
  Conjecture \ref{con:Hn1} is slightly weaker than its analogue
  Proposition \ref{pro:factor}; replacing the character
  $(-1)^{\ell(w)}$ by $Y^{\ell(w)}$ on the left hand side does not
  lead to a similar factorization. Note that
  \cite[Proposition~5.5]{StasinskiVoll/14} and
  \cite[Theorem~5.4]{BrentiCarnevale/17} yield another factorization
  formula for the left-hand side which is not, however, expressed in
  terms of statistics on the symmetric group. Moreover,
  \cite[Lemma~8]{StasinskiVoll/13} implies that the sum on the
  left-hand side remains unchanged when restricted to chessboard
  elements; for definitions and properties of these see
  \cite{StasinskiVoll/13} and \cite{BrentiCarnevale/17}.
\end{remark}

\subsection{Igusa functions}
Recall, e.g.\ from \cite[Definition~2.5]{SV1/15}, the definition of
the \emph{Igusa function} (of degree $n$)
\begin{align*}
  I_n(Y;X_1,\dots,X_n) &= \frac{1}{1-X_n} \sum_{I\subseteq[n-1]}
  \binom{n}{I}_Y \prod_{i\in I} \gp{(X_i)} \nonumber\\ &=
  \frac{\sum_{w\in S_n} Y^{\ell(w)}\prod_{j\in \Des(w)}
    X_j}{\prod_{j=1}^n(1-X_j)}\in\Q(Y,X_1,\dots,X_n). \label{def:igusa}
\end{align*}
Specific choices of ``numerical data'' to be substituted for the
variables $Y,X_1,\dots,X_n$ may lead to factorizations or
cancellations. An extremal example is the following.

\begin{example}\label{exa:abelian}\cite[Proposition 4.2]{StasinskiVoll/14}
\begin{equation}\label{exa:total.factor}
  I_n\left(X^{-1}; \left(X^iZ)^{n-i}\right)_{i=n-1}^{0}\right) = \frac{1}{\prod_{j=0}^{n-1}(1-X^jZ)}.
\end{equation}
\end{example}

We record an application to the rational function
$$\mcB_n(X,Y,Z) = \sum_{j=0}^n
\binom{n}{j}_X\frac{Z^{n-j}\prod_{i=0}^{n-j-1}(1-X^{-i-j-1}Y)}{\prod_{i=0}^{n-j-1}(1-X^{i+j}Z)}$$
defined in \cite[eq.~(4.7)]{StasinskiVoll/14}. 
\begin{corollary}\label{cor:Bn}
  \begin{multline*}\mcB_n(X,-Y,X^nZ) =    \frac{\prod_{j=0}^{n-1}(1 + X^jYZ)}{\prod_{j=0}^{n-1}\left(1 -
        X^{n+j}Z\right)} \\=I_n\left(X^{-1},
      \left(X^{n^2-i^2}Z^{n-i}\right)_{i=n-1}^0\right)
    \prod_{j=0}^{n-1}\left(1 + X^jYZ\right).
\end{multline*}
\end{corollary}
\begin{proof}
  The first equality is proven in
  \cite[Section~4.2.2]{StasinskiVoll/14}, the second follows
  from~\eqref{exa:total.factor}.\qedhere
\end{proof}

\subsection{Traceless matrices, signed permutation statistics, and Igusa functions}\label{subsec:weyl.group}
We recast the formula for the image zeta function
$\mathcal{P}_{n,\lri}(s)$ given in Theorem~\ref{thm:Pn.first} in terms
of signed permutation statistics;
cf.\ Proposition~\ref{pro:Pn.second}. In Theorem~\ref{thm:Pn.factor}
we establish an expression for $\mathcal{P}_{n,\lri}(s)$ in terms of
Igusa functions.

\begin{proposition}\label{pro:Pn.second}
  The following identities hold in the field
  $\Q(Y,X_{0},\dots,X_{n-1})$:
  \begin{multline}\label{eq:PnB}
    1 + \sum_{i=0}^{n-1} b_{n,i}(Y) \gp(X_{i}) \left(
      \sum_{J\subseteq [i-1]_0} f_{{i},J}(Y) \prod_{j\in
        J}\gp(X_{j})\right) = \\\sum_{I\subseteq [n-1]_0}
    f_{K_n,I}(Y)\prod_{j \in I}\gp(X_{j}) = \frac{\sum_{w\in
        B_n}(-1)^{\nega(w)}Y^{(\ell - \nch)(w)}\prod_{j \in \Des(w)}
      {X_{j}}}{\prod_{j=0}^{n-1}(1-X_{j})}.\end{multline}
\end{proposition}

\begin{proof}
Eq.\ \eqref{equ:fact} yields
\begin{multline*}\label{eq:Pn2} {\sum_{I\subseteq [n-1]_0}
    f_{K_n,I}(Y)\prod_{j \in I}\gp(X_{j})}={1+\sum_{i=0}^{n-1} \sum_{
      i \in J\subseteq [i]_0}f_{K_n,J}(Y)\prod_{j\in
      J}\gp(X_{j})}\\=1+ \sum_{i=0}^{n-1} b _{n,i}(Y) \gp(X_{i})
  \left( \sum_{J\subseteq [i-1]_0} f_{i,J}(Y) \prod_{j\in
    J}\gp(X_{j})\right),
\end{multline*}
establishing the first equality.  The second one follows from
\eqref{equ:fI} and \cite[Lemma~4.4]{StasinskiVoll/14}.
\end{proof}

Theorem~\ref{thm:Pn.first} shows that the Poincar\'e series
$\mathcal{P}_{n,\lri}(s)$ may be obtained from the rational function
on the left-hand side of \eqref{eq:PnB} by substituting $q^{-1}$ for
the variable $Y$ and $x_{n,i} = q^{n^2-i^2-1}t^{n-i}$ for the
variables $X_{i}$. Our next result shows that, under this
substitution, the numerator of the rational function on the right-hand
side of \eqref{eq:PnB} factorizes partly.

\begin{theorem}\label{thm:Pn.factor}
  \begin{equation}\label{equ:Pn.factor}
  \mathcal{P}_{n,\lri}(s) = I_n\left(q^{-1};\left( q^{n^2-i^2-1}t^{n-i}\right)_{i=n-1}^0\right)\prod_{j=0}^{n-1}(1-q^{j}t).
\end{equation}
\end{theorem}

\begin{proof}
  Using Theorem~\ref{thm:Pn.first} and
  Propositions~\ref{pro:Pn.second} and \ref{pro:factor} (where we
  substitute $(q^{-1},q,-1,t)$ for $(W,X,Y,Z)$), we obtain
\begin{align*}
  \mathcal{P}_{n,\lri}(s) &= 1 + \sum_{i=0}^{n-1} b_{n,i}(q^{-1})
                            \gp(x_{n,i}) \left( \sum_{J\subseteq [i-1]_0} f_{i,J}(q^{-1})
                            \prod_{j\in J}\gp(x_{n,j})\right)\\ 
                          &= \frac{\sum_{w\in
                            B_n}(-1)^{\nega(w)} q^{(-\ell + \nch)(w)} \prod_{j\in \Des(w)}
                            q^{n^2-j^2-1}t^{n-j}}{\prod_{j=0}^{n-1}(1-q^{n^2-j^2-1}t^{n-j})}
  \\ &= \frac{\sum_{w\in B_n}(-1)^{\nega(w)}q^{(\sigma_B -\ell +
       \nch -
       \des)(w)}t^{\rmaj(w)}}{\prod_{j=0}^{n-1}(1-q^{n^2-j^2-1}t^{n-j})}\\ &=
                                                                             \frac{\left(\sum_{w\in S_n}q^{(\sigma_A-\ell-\des)(w)}(q^n
                                                                             t)^{\rmaj(w)}\right)\prod_{j=0}^{n-1} (1-q^j
                                                                             t)}{\prod_{j=0}^{n-1}(1-q^{n^2-j^2-1}t^{n-j})}\\ &=
                                                                                                                                \frac{\left(\sum_{w\in S_n}q^{-\ell(w)} \prod_{j\in \Des(w)}
                                                                                                                                q^{j(n-j)-1+n(n-j)}t^{n-j}\right)\prod_{j=0}^{n-1} (1-q^j
                                                                                                                                t)}{\prod_{j=0}^{n-1}(1-q^{n^2-j^2-1}t^{n-j})}\\ &=
                                                                                                                                                                                   I_n\left(q^{-1};\left(
                                                                                                                                                                                   q^{n^2-i^2-1}t^{n-i}\right)_{i=n-1}^0\right)\prod_{j=0}^{n-1}(1-q^{j}t). \qedhere
\end{align*}
\end{proof}

\begin{remark}
  The inverse of the second factor $\prod_{j=0}^{n-1}(1-q^{j}t)$ on
  the right-hand side of \eqref{equ:Pn.factor} is itself an Igusa
  function of degree~$n$,
  viz.\ $I_n(q^{-1};\left((q^{i}t)^{n-i}\right)_{i=n-1}^0)$;
  cf.\ Example~\ref{exa:abelian}. Numerical evidence for small $n$
  suggests that there is no further ``systematic'' factorization of
  the first factor; cf.\ Example~\ref{exa:small.n}.
\end{remark}

\section{Representation zeta functions of groups of type
  $K$}\label{sec:rep.zeta.nilpotent}

In this section we consider applications of the formulae obtained in
the previous sections to representation zeta functions of finitely
generated nilpotent groups of type~$K$, in particular their global
analytic properties. In the sequel we assume that~$n>1$, ensuring that
the group schemes $K_n$ are nonabelian.

\subsection{Global analytic properties of Euler products} We consider
Euler products of Poincar\'e series of the form $\mcP_{n,\lri}(s)$,
where $\lri$ runs through the completions $\mcO_\mfp$ of the ring of
integers $\mcO$ of a number field $F$ at its nonzero prime ideals
$\mfp$. As mentioned in Section~\ref{subsec:app.T}, the (global)
representation zeta function of the group $K_n(\Gri)$ satisfies
\begin{equation*}\label{eq:eprod}
  \zeta_{K_n(\Gri)}(s) = \prod_{\mfp \in\Spec(\Gri)}
  \zeta_{K_n(\Gri_\mfp)}(s) = \prod_{\mfp \in\Spec(\Gri)}
  \mathcal{P}_{n,\Gri_\mfp}(s).
\end{equation*}

\begin{proposition}\label{pro:anprop} 
\leavevmode \begin{enumerate}
\item The abscissa of convergence of $\zeta_{K_n(\Gri)}(s)$ is
  $\alpha(K_n(\Gri)) = 2n-1$.
\item The zeta function $\zeta_{K_n(\Gri)}(s)$ may be continued
  meromorphically to $\{s \in \C \mid \Re(s)>2n-3\}$. For
  $n\in\{2,3\}$ it may be continued meromorphically to the whole
  complex plane.
\end{enumerate}
\end{proposition}

\begin{proof}
  According to Theorem~\ref{thm:Pn.factor}, the representation zeta
  function $\zeta_{K_n(\Gri)}(s)$ is the product of a finite number of
  (inverses of) translates of the Dedekind zeta function $\zeta_F(s)$
  of the number field $F$ and an Euler product of Igusa
  functions. Equivalently, writing
$$I_n\left(q^{-1};\left(
    q^{n^2-i^2-1}t^{n-i}\right)_{i=n-1}^0\right) =
\frac{V_n(q,t)}{\prod_{j=0}^{n-1}(1-q^{n^2-j^2-1}t^{n-j})}$$ with
$$V_n(q,t)=\sum_{w\in S_n} q^{-\ell(w)}\prod_{j\in
  \Des(w)}{q^{n^2-j^2-1}t^{n-j}}$$ (cf.\ Theorem~\ref{thm:main}), we
have
$$\zeta_{K_n(\Gri)}(s) = \left(
\prod_{j=0}^{n-1}\frac{\zeta_{F}((n-j)s-(n^2-j^2-1))}{\zeta_F(s-j)}\right)\prod_{\mfp\in\Spec(\Gri)}
V_n(q,t).$$ Here as throughout, $q=|\Gri:\mfp|$ denotes the residue
field cardinality at $\mfp\in\Spec(\Gri)$.  Recall that, for
$a,b\in\R$, with $a\neq 0$, the translate $\zeta_F(as-b)$ converges
uniformly on $\{s\in\C \mid \Re(s)>\frac{b+1}{a}\}$ and has
meromorphic continuation to the whole complex plane. Obviously, $\max
\left\{ \frac{n^2-j^2-1+1}{n-j} \mid j\in[n-1]_0 \right\} =
2n-1$. Given the explicit formulae for $\zeta_{K_n(\Gri_{\mfp})}(s)$
for $n\in\{2,3\}$ in Example~\ref{exa:small.n}, the proposition thus
holds for $n\in\{2,3\}$ and we may assume that $n\geq 4$. In this
general case, the proposition will follow if we can show that the
Euler product
\begin{equation*}\label{equ:euler.numer}
\prod_{\mfp\in\Spec(\Gri)} V_n(q,t)
\end{equation*}
may be meromorphically continued to $\{s \in \C \mid \Re(s)>2n-3\}$.

By \cite[Lemma~5.5]{duSWoodward/08}, it suffices to prove that
\begin{equation*}\label{claim.beta}
  \beta := \max \left\{\frac{\left(\sum_{j\in \Des(w)}(n^2-j^2-1)\right) - \ell(w)}{\sum_{j\in \Des(w)}(n-j)} \mid w\in S_n \setminus \{1\}\right\} = 2n-3.
\end{equation*}
We will prove, more precisely, that this maximum is attained exactly
twice, viz.\ at the Coxeter generators $w=s_{n-1}$ and $w=s_{n-2}$. We
first show that, for $w\in S_n\setminus \{1\}$, the quantity
\begin{equation*}
{\frac{\left(\sum_{j\in
          \Des(w)}(n^2-j^2-1)\right)-\ell(w)}{\sum_{j\in
        \Des(w)}(n-j)}}
\end{equation*}
does not exceed $2n-3$. Indeed, suppose that it is $\geq 2n-3$, i.e.\
\begin{align*}
  \left(\sum_{j\in \Des(w)}((n-j)(n+j)-1) \right)-\ell(w)-(2n-3)\sum_{j\in
  \Des(w)}(n-j)& = \\ \left(\sum_{j\in
  \Des(w)}(-(n-j)^2+3(n-j))\right)-\des(w)-\ell(w) &\geq 0.
\end{align*}
Writing $i=n-j$ and setting $\rho(w) = \ell(w)-\des(w)\in\N_0$, this is
equivalent to
\begin{equation*}\label{eq:last} \sum_{n-i \in \Des(w)}
  (3i-2)\geq \left(\sum_{n-i \in \Des(w)} i^2\right) + \rho(w).
\end{equation*}
The latter inequality has solutions only if $\Des(w)\subseteq
\{n-2,n-1\}$. It is an equality if, and only if,
additionally~$\rho(w)=0$. The elements $w=s_{n-1}$ and $w=s_{n-2}$
clearly satisfy this condition and are the only solutions whose
descent sets are singletons; no permutation with $\Des(w)=\{n-2,n-1\}$
satisfies $\rho(w)=0$.

We conclude that the maximum $\beta = 2n-3 $ is attained at $s_{n-1}$
and $s_{n-2}$ as claimed.
\end{proof}
\begin{remark}
  Proposition \ref{pro:anprop} leaves open the interesting question
  whether or not the line $\{ s\in \C \mid \Re(s) = 2n-3\}$ is
  actually a natural boundary for meromorphic continuation of
  $\zeta_{K_n(\Gri)}(s)$ for $n\geq 4$.  It may be of interest to note
  that the proof of its part~(2) implies that, in the terminology of
  \cite[Section~5.2]{duSWoodward/08}, the first factor of the ghost
  polynomial of $V_n(q,t)$ is the unitary polynomial
  $\widetilde{V_{n,1}}(q,t)=1+q^{2n-3}t+q^{4n-6}t^2$.

  \rev{That $\zeta_{K_n(\Gri)}(s)$ has abscissa of convergence which is
  independent of $\Gri$ and admits some meromorphic continuation into
  a region defined, again, independently of $\Gri$ are special cases
  of general results on representation zeta functions of nilpotent
  groups; cf.\ \cite[Theorem~A]{DungVoll/17}.}
\end{remark}

\subsection{Topological representation zeta functions}\label{subsec:top}

In \cite{Rossmann_top_uni/16}, Rossmann initiated the study of
topological representation zeta functions associated to unipotent
group schemes defined over number fields. Very roughly speaking, these
are rational functions encapturing the ``limit $q\rarr 1$'' of the
local representation zeta functions occurring as Euler factors in the
global representation zeta functions associated to the groups of
rational points over number rings of the unipotent group schemes in
question. The latter are rational functions in the parameters~$q^{-s}$
and the topological representation zeta function may be defined as the
leading coefficients of the series expansions of these rational
functions in~$q-1$. For precise definitions and instructive examples
see \cite[Section~3]{Rossmann_top_uni/16}. A number of intriguing open
questions regarding topological representation zeta functions of
unipotent group schemes are raised in
\cite[Section~7]{Rossmann_top_uni/16}. We use the explicit formulae
derived in the current paper to show that some of these questions have
positive answers for the unipotent group schemes~$K_n$. The following
is an immediate consequence of Theorem~\ref{thm:Pn.factor}.

\begin{proposition}\label{pro:topo}
  The topological zeta function of $K_n$ is equal to
$$\zeta_{K_n,\textup{top}}(s) = \prod_{i=0}^{n-1}
  \frac{s-i}{s-(n+i-\frac{1}{n-i})} .$$
\end{proposition}

We note that $\zeta_{K_n,\textup{top}}(s)$ has a simple zero at $s=0$
and that $\zeta_{K_n,\textup{top}}(s)-1$ has degree $-1$ in $s$;
cf.\ \cite[Questions~7.4 and~7.1]{Rossmann_top_uni/16}. Following
\cite{Rossmann_top_uni/16} we consider the invariant
 $$\omega(K_n)=s\left(\zeta_{K_n,\textup{top}}(s)-1\right)\big\vert_{s=\infty}
\in \Q.$$ In the pertinent special cases,
\cite[Question~7.2]{Rossmann_top_uni/16} is answered positively by the
following result.

\begin{proposition} 
  \rev{$$\omega(K_n) =n^2-H_n,$$ where $H_n = \sum_{k=1}^n
    \frac{1}{k}$ is the $n$-th harmonic number.}
\end{proposition}
\begin{proof}
  The quantity $\omega(K_n)$ is the quotient of the leading
  coefficients of the polynomials in $s$ occurring in numerator and
  denominator of the right-hand side of the expression
  $$ \zeta_{K_n,\textup{top}}(s)-1 = \frac{ \left(\prod_{i=0}^{n-1}
      (s-i)\right) -\left(\prod_{i=0}^{n-1}
      \left(s-\left(n+i-\frac{1}{n-i}\right)
      \right)\right)}{\prod_{i=0}^{n-1}
    \left(s-\left(n+i-\frac{1}{n-i}\right) \right)}.$$
  As the denominator is monic, $\omega(K_n)$ is in fact just the
  difference between the coefficients of $s^{n-1}$ in
  $\prod_{i=0}^{n-1} (s-i)$ and
  $\prod_{i=0}^{n-1} \left(s-\left(n+i-\frac{1}{n-i}\right) \right)$,
  respectively.
  \rev{Hence
    $$\omega(K_n) = -\binom{n}{2}+
    \sum_{i=0}^{n-1}\left(n+i-\frac{1}{n-i}\right) = \binom{2n}{2} - 2
    \binom{n}{2} - H_n = n^2-H_n.\qedhere$$}
\end{proof}

\begin{remark}
 The topological representation zeta functions the unipotent group
 schemes $F_{n,\delta}$, $G_n$, and $H_n$ are easily read off from
 \cite[Theorem~B]{StasinskiVoll/14}:
 \begin{align*}
\zeta_{F_{n,\delta},\textup{top}}(s)& = \prod_{i=0}^{n-1}
\frac{s-2i}{s-2(n+i+\delta) + 1} ,\\ \zeta_{G_n,\textup{top}}(s) &= \prod_{i=0}^{n-1} \frac{s-i}{s-
  n-i} ,\\ \zeta_{H_n,\textup{top}}(s)& = \prod_{i=0}^{n-1}
\frac{s-i}{s- \frac{n+i+1}{2}} .
 \end{align*}  
 \cite[Questions~7.1 and 7.4]{Rossmann_top_uni/16} are easily seen to
 have positive answers in these cases, too. One computes easily that
 \[\omega(F_{n,\delta})=2n^2+(2\delta-1)n,
 \qquad \omega(G_n)=n^2, \qquad \omega(H_n)= \frac{n^2+3n}{4}. \] In
 contrast, the invariant $\omega(K_n)$ \rev{is never an integer;
   cf.\ \cite{Theisinger/15}.}
 \end{remark}

\subsection{Representation zeta functions and Igusa functions}

Theorem~\ref{thm:Pn.factor} describes local representation zeta
functions associated to groups of type $K$ as quotients of two Igusa
functions. We note similar factorizations for groups of type $F$ and
$G$ and observe that Conjecture~\ref{con:Hn1} yields an analogous
conjectural expression for groups of type~$H$.

\begin{proposition} Let $\delta\in\{0,1\}$ and $\lri$ be of
  characteristic zero. The representation zeta functions
  $ \zeta_{F_{n,\delta}(\lri)}(s)$ and $ \zeta_{G_{n}(\lri)}(s)$
  described in \cite[Theorem~C]{StasinskiVoll/14} satisfy
  \begin{align}
 \zeta_{F_{n,\delta}(\lri)}(s) &=
 I_n\left(q^{-2};\left(q^{\binom{2n+\delta}{2} -
   \binom{2i+\delta}{2}}t^{n-i}\right)_{i=n-1}^{0}\right)
 \prod_{i=0}^{n-1}(1-q^{2i}t),\label{equ:Igusa.F}\\ \zeta_{G_n(\lri)}(s)
 &=
 I_n\left(q^{-1};\left(q^{n^2-i^2}t^{n-i}\right)_{i=n-1}^0\right)\prod_{i=0}^{n-1}(1-q^it)\label{equ:Igusa.G}
\end{align}
\end{proposition}

\begin{proof}
  Using that
  $\zeta_{F_{n,\delta}(\lri)}(s) = \mcB_{n}(q^2,
  q^{-2\delta+1},q^{2(n+\delta)-1-s})$ and
  $\zeta_{G_n(\lri)}(s) = \mcB_{n}(q,1,q^{n-s})$ (cf.\ proof of
  \cite[Proposition~5.1]{StasinskiVoll/14}), this follows from
  Corollary~\ref{cor:Bn}.
\end{proof}

\begin{conjecture}\label{con:Hn}
  Let $\lri$ be of characteristic zero. The representation zeta
  function $ \zeta_{H_{n}(\lri)}(s)$ described in
  \cite[Theorem~C]{StasinskiVoll/14} satisfies
  \begin{equation}\label{equ:Igusa.H}\zeta_{H_n(\lri)}(s) =
  I_n\left(q^{-1};\left(q^{\binom{n+1}{2}-\binom{i+1}{2}}t^{n-i}\right)_{i=n-1}^0\right)\prod_{i=0}^{n-1}(1-q^it).\end{equation}
\end{conjecture} 

Regarding the ``numerical data'' of the Igusa functions in
\eqref{equ:Igusa.F}, \eqref{equ:Igusa.G}, and \eqref{equ:Igusa.H}, we
remark that, in the notation of \cite[Theorem~C]{StasinskiVoll/14}.
\begin{equation*}
  {\binom{2n+\delta}{2} - \binom{2i+\delta}{2}} = a(F_{n,\delta},i),\quad
  n^2-i^2 = a(G_n,i), \quad
  \binom{n+1}{2}-\binom{i+1}{2} = a(H_n,i).
\end{equation*}

\begin{acknowledgements}
  \rev{We are grateful for an anonymous referee's helpful comments, in
    particular about harmonic numbers.}  We acknowledge support by the
  German Research Council (DFG) through Sonderforschungsbereich 701 at
  Bielefeld University. Carnevale and Voll were partly supported by
  the German-Israeli Foundation for Scientific Research and
  Development (GIF) through grant no.~1246. Shechter was partially
  supported by the Israel Science Foundation (ISF) through grant
  no.~1862.
\end{acknowledgements}

\def\cprime{$'$}
\providecommand{\bysame}{\leavevmode\hbox to3em{\hrulefill}\thinspace}
\providecommand{\MR}{\relax\ifhmode\unskip\space\fi MR }
\providecommand{\MRhref}[2]{%
  \href{http://www.ams.org/mathscinet-getitem?mr=#1}{#2}
}
\providecommand{\href}[2]{#2}

\end{document}